\documentclass[12pt,reqno]{amsart}

\usepackage[l2tabu,orthodox]{nag}

\usepackage{hyperref}

\usepackage{enumerate}

\usepackage{cmap}

\usepackage[swedish,english]{babel}
\usepackage[T1]{fontenc}
\usepackage[latin1]{inputenc}

\usepackage{a4wide}

\usepackage[pdftex,dvipsnames]{xcolor}

\usepackage{amsmath,amssymb,amsthm}

\theoremstyle{plain}
\newtheorem{thm}{Theorem}[section]
\newtheorem{lem}[thm]{Lemma}
\newtheorem{prop}[thm]{Proposition}
\newtheorem{cor}[thm]{Corollary}
\theoremstyle{definition}

\newtheorem{exmp}[thm]{Example}
\theoremstyle{remark}
\newtheorem{rem}[thm]{Remark}

\theoremstyle{plain}
\newtheorem*{BurnsThm}{Theorem}

\theoremstyle{plain}
\newtheorem{mainthm}{Theorem}

\theoremstyle{plain}
\newtheorem{mainCorA}{Corollary}

\theoremstyle{plain}
\newtheorem{mainCorB}{Corollary}[mainthm]

\newcommand{\Z}{\mathbb{Z}}
\newcommand{\Q}{\mathbb{Q}}

\newcommand{\C}{\mathbb{C}}
\renewcommand{\emptyset}{\varnothing}

\DeclareMathOperator{\Endr}{Endr}
\DeclareMathOperator{\Supp}{Supp}

\title[Central idempotents in group-graded rings]{Central idempotents in group-graded rings}
\date{\today}

\begin{document}

\author{Johan \"{O}inert}
\address{
Department of Mathematics and Natural Sciences,
Blekinge Institute of Technology,
SE-37179 Karlskrona, Sweden
and
Department of Engineering, University of Sk\"{o}vde, SE-54128 Sk\"{o}vde, Sweden
}
\email{johan.oinert@bth.se}


\begin{abstract}
Let $G$ be a group and let $R$ be a $G$-graded ring. 
We show that every nonzero central idempotent in $R$ has finite support group in two broad settings:
when $G$ is abelian, and when $G$ is arbitrary but the grading satisfies a certain one-sided non-annihilation condition on nonzero homogeneous elements.
In particular, 
under the respective hypotheses,
if $G$ is torsion-free, then every central idempotent lies in the principal component of the grading.
Our results generalize earlier results by H. Bass, R. G. Burns, and A. A. Bovdi--S. V. Mihovski, from group rings and crossed products, to non-commutative, possibly non-unital, group-graded rings.
We demonstrate the utility of our results by applying them to semigroup-graded rings, Leavitt path rings, fractional skew monoid rings, partial skew group rings, and algebraic Cuntz-Pimsner rings.
\end{abstract}

\maketitle

\section{Introduction}

Let $G$ be a group with identity element $e$.
Recall that a (possibly non-unital) ring $R$ is said to be \emph{$G$-graded} if there is a collection of additive subgroups $\{R_g\}_{g\in G}$ of $R$
such that $R=\oplus_{g\in G} R_g$, and $R_g R_h \subseteq R_{gh}$ for all $g,h\in G$.
If $R_g R_h = R_{gh}$ holds for all $g,h\in G$, then $R$ is said to be \emph{strongly $G$-graded}.
The \emph{support} of a nonzero element $r=\sum_{g\in G} r_g$ is the finite set $\Supp(r):=\{g\in G \mid r_g \neq 0\}$.
The subgroup of $G$ that is generated by $\Supp(r)$ is said to be \emph{the support group of $r$}.
Although finitely generated, the support group of an element need not be finite.

A group ring is a prototypical example of a strongly group-graded ring.
In 1970, Burns~\cite{Burns} established the following result regarding the support group of central idempotents in group rings,
building on earlier work by Rudin and Schneider \cite{RudinSchneider}.

\begin{BurnsThm}[Burns]
Let $A$ be a unital ring and let $G$ be a group.
If $f$ is a nonzero central idempotent in the group ring $A[G]$,
then $f$ has finite support group.
\end{BurnsThm}

In fact, Burns proved the theorem for non-unital group rings, but we only cite the unital version here
since it will be enough for our purposes and also consistent with how group rings are usually defined nowadays.
Independently, 
Bovdi and Mihovski \cite{BovdiMihovski,BovdiMihovski1973} 
showed that the above conclusion 
holds for any (unital) $G$-crossed product.
Later on, Bass \cite[Lem.~6.7]{Bass} provided 
a short and elegant proof that 
every idempotent in a commutative group ring $A[G]$ is contained 
in the subring $A[T]$, where $T$ denotes the torsion subgroup of $G$.

In this article, we use Burns's theorem as a lever, allowing us
to prove similar statements in the more general context of group-graded rings.
In Section~\ref{Sec:First}, we establish 
our first main result which generalizes
Bass's result
from the setting of commutative group rings to the setting of group-graded rings that are 
allowed to be both
non-commutative and non-unital.

\begin{mainthm}\label{thm:Main1}
Let $G$ be an abelian group and let $R$ be a $G$-graded ring.
If $f$ is a nonzero central idempotent in $R$,
then $f$ has finite support group.
In particular, the support of a nonzero central idempotent in $R$
is contained in the torsion subgroup of $G$.
\end{mainthm}

As a direct consequence, we obtain a partial generalization of 
a result of Kirby 
\cite[Thm.~1]{Kirby} that was proved in the setting of strongly group-graded commutative rings.

\begin{mainCorA}\label{cor:AbelianTorFree}
Let $G$ be a torsion-free abelian group and let $R$ be a 
$G$-graded ring.
Then every central idempotent in $R$ is contained in $R_e$.
\end{mainCorA}

Recall that an element of a $G$-graded ring $R$ is said to be \emph{homogeneous}
if it is contained in the set $\cup_{g\in G} R_g$.
Also recall that we write $\Supp(R):=\{g\in G \mid R_g \neq \{0\}\}$.
In Section~\ref{Sec:Second}, we establish our second main result
which is the following generalization of Burns's theorem.

\begin{mainthm}\label{thm:Main2}
Let $G$ be a group and let $R$ be a $G$-graded ring.  
Suppose that 
at least one of the following two assertions holds:
\begin{enumerate}[{\rm (i)}]
	\item $R_gr \neq \{0\}$
	for all nonzero homogeneous $r\in R$ and all $g\in \Supp(R)$.
	\item $rR_g \neq \{0\}$
	for all nonzero homogeneous $r\in R$ and all $g\in \Supp(R)$.
\end{enumerate}
If $f$ is a nonzero central idempotent in $R$,
then $f$ has finite support group.
\end{mainthm}

We point out that 
Theorems~\ref{thm:Main1} and \ref{thm:Main2} 
cannot be generalized to hold for arbitrary idempotents (see Example~\ref{ex:NonExample1}).
We also demonstrate that the conclusions of Theorems~\ref{thm:Main1} and \ref{thm:Main2} may fail if we weaken the assumptions (see Example~\ref{ex:Failure1}).

Theorem~\ref{thm:Main2} yields 
the following three corollaries that 
we prove
at the end of Section~\ref{Sec:Second}.
Recall that a ring $R$ is said to be \emph{$s$-unital} if $r \in rR \cap Rr$ for every $r\in R$.
The following corollary applies in particular to any unital strongly $G$-graded ring.

\begin{mainCorB}\label{cor:Strongly}
Let $G$ be a group and let $R$ be an $s$-unital strongly $G$-graded ring.
If $f$ is a nonzero central idempotent in $R$,
then $f$ has finite support group.
In particular, if $G$ is torsion-free, then every central idempotent in $R$ is contained in $R_e$.
\end{mainCorB}

Note that the preceding corollary and the next corollary both apply to any $G$-crossed product (see e.g. \cite{NVO} and \cite{NauwVO}).

\begin{mainCorB}\label{cor:Crystalline}
Let $G$ be a group and let $A*G$ be a crystalline $G$-graded ring.
If $f$ is a nonzero central idempotent in $A*G$,
then $f$ has finite support group.
In particular, if $G$ is torsion-free, then every central idempotent in $A*G$ is contained in $A$.
\end{mainCorB}

A $G$-graded ring $R$ is said to be \emph{right non-degenerately $G$-graded} (resp. \emph{left non-degenerately $G$-graded}) if,
for each $g \in G$ and each nonzero $r_g \in R_g$, we have $r_gR_{g^{-1}} \neq \{0\}$ (resp. $R_{g^{-1}} r_g \neq \{0\}$).
Notably, the next corollary generalizes \cite[Thm.~6.2(iii)]{Oinert}.

\begin{mainCorB}\label{cor:NonDeg}
Let $G$ be a group and let $R$ be a $G$-graded ring. 
Suppose that $R_e$ is a prime ring and that $R$ is right (or left) non-degenerately $G$-graded.
If $f$ is a nonzero central idempotent in $R$,
then $f$ has finite support group.
In particular, if $G$ is torsion-free, then every central idempotent in $R$ is contained in $R_e$.
\end{mainCorB}

In Section~\ref{Sec:Applications}, we apply our results to semigroup-graded rings, Leavitt path rings, fractional skew monoid rings, partial skew group rings, and algebraic Cuntz-Pimsner rings.

\section{Proof of Theorem~\ref{thm:Main1}}\label{Sec:First}

Given a unital ring $A$ and a group $G$, the group ring $A[G]$ is defined as the free left (and right) 
$A$-module with $\{u_g\}_{g\in G}$ as its basis. Multiplication is defined by extending the rule
$(a u_g)(b u_h) := ab u_{gh}$, for $a,b\in A$ and $g,h\in G$.
Note that $R:=A[G]$ carries a canonical $G$-grading that is given by $R_g := A u_g$, for $g\in G$.

The proof of Theorem~\ref{thm:Main1} uses a trick based on
the following graded embedding that is inspired by \cite[p.~327]{LeePuczylowski}.

\begin{lem}\label{lem:EmbeddingTrick}
Let $G$ be a group and let $R$ be a unital $G$-graded ring.
Equip the group ring $R[G]$ with its canonical $G$-grading.
There is an injective identity-preserving ring homomorphism $\phi : R \to R[G]$ 
satisfying $\Supp(r)=\Supp(\phi(r))$ for every $r\in R$.
\end{lem}

\begin{proof}
Let $r=\sum_{g\in G} r_g \in R$, with $r_g \in R_g$ zero for all but finitely many $g\in G$.
We define $\phi : R \to R[G]$ by $\phi\left( \sum_{g\in G} r_g \right) := \sum_{g\in G} r_g u_g$.
The map $\phi$ is clearly well defined, additive and injective.

Take $g,h \in G$, $a \in R_g$, and $b \in R_h$.
Using that $a b \in R_g R_h \subseteq R_{gh}$, we note that
\begin{displaymath}
	\phi(ab) = ab u_{gh} = (a u_g)(b u_h) = \phi(a) \phi(b).
\end{displaymath}
It follows that $\phi$ is multiplicative.
By \cite[Prop.~1.1.1]{NVO}, 
$\phi(1_R)=1_R u_e=1_{R[G]}$.
Clearly, $\Supp(r)=\Supp(\phi(r))$.
\end{proof}

For each $g\in G$, we define the map $(\cdot)_g : R \to R_g, \ \sum_{h\in G} r_h \mapsto r_g$. 
We let $Z(R)$ denote the center of $R$.

\begin{lem}\label{lem:CentralComponentWise}
Let $G$ be an abelian group and let $R$ be a $G$-graded ring.
If $f \in Z(R)$, then $f_g \in Z(R)$, for every $g\in G$.
\end{lem}

\begin{proof}
Suppose that $f\in Z(R)$.
Take $h\in G$ and $r_h\in R_h$. Then, for any $g\in G$, we get that $r_h f_g = (r_h f)_{hg} = (f r_h)_{gh} = f_g r_h$.
It now follows that $f_g \in Z(R)$, for every $g\in G$.
\end{proof}

\begin{rem}\label{rem:Dorroh}
Let $D$ be the Dorroh unitization of $R$. Recall that $D := R \times \Z$ is an associative and unital ring 
with addition defined componentwise and multiplication defined by $(r,n)(s,m):=(rs+ns+mr,nm)$ 
for all $r,s \in R$
and
for all $n,m\in \Z$
(for more details, see Dorroh's original article \cite{Dorroh}).

Note that there is an injective ring homomorphism $\psi : R \to D$ defined by $\psi(r):=(r,0)$.
Furthermore, note that if $f$ is a nonzero central idempotent in $R$, then $\psi(f)$ is a nonzero central idempotent in $D$.

Suppose that $R$ is a $G$-graded ring. Then $D$ inherits a natural $G$-grading from $R$.
Indeed, define $D_e := R_e \times \Z$ and $D_g := R_g \times \{0\}$ for every $g\in G \setminus \{e\}$.
Then $D$ becomes a unital $G$-graded ring.
Importantly, note that the support group of $f$ in the $G$-graded ring $R$
equals the support group of $f$ in the unital $G$-graded ring $D$.
Therefore, for the purposes of proving Theorem~\ref{thm:Main1}, after identifying $R$ with its image in $D$,
we may always view $f$ as an element of a unital $G$-graded ring. 
\end{rem}

We are now ready to prove our first main result.

\begin{proof}[Proof of Theorem~\ref{thm:Main1}]
By Remark~\ref{rem:Dorroh}, replacing $R$ by its Dorroh unitization if necessary, we may assume that $R$ is unital.
Let $f$ be a nonzero central idempotent in $R$.
By Lemma~\ref{lem:EmbeddingTrick}, $\phi(f)$ is a nonzero idempotent in the group ring $R[G]$,
and $\Supp(f)=\Supp(\phi(f))$.

Take $h\in G$ and $r \in R$. By Lemma~\ref{lem:CentralComponentWise} and the fact that $G$ is abelian,
we get that
\begin{displaymath}
	r u_h \phi(f) = r u_h \left(\sum_{g\in G} f_g u_g\right) = \sum_{g\in G} rf_g u_{hg} = \left(\sum_{g\in G} f_g u_g \right) r u_h = \phi(f) r u_h.
\end{displaymath}
From this we conclude that $\phi(f) \in Z(R[G])$.

It now follows immediately from Burns's theorem that $\phi(f)$, and hence also $f$, has finite support group.
Since $G$ is abelian, the set of torsion elements forms a subgroup $T$ of $G$, and clearly the support group of $f$ is contained in $T$.
\end{proof}

\section{Proof of Theorem~\ref{thm:Main2}}\label{Sec:Second}

Recall that a group is said to be an \emph{FC-group} if each of its conjugacy classes is finite.
Next, we establish the following generalization of \cite[Prop.~6.1]{Oinert}.

\begin{lem}\label{lem:centralFCsupportGeneralized}
Let $G$ be a group and let $R$ be a $G$-graded ring. 
Suppose that 
at least one of the following two assertions holds:
\begin{enumerate}[{\rm (i)}]
	\item $R_gr \neq \{0\}$
	for all nonzero homogeneous $r\in R$ and all $g\in \Supp(R)$.
	\item $rR_g \neq \{0\}$
	for all nonzero homogeneous $r\in R$ and all $g\in \Supp(R)$.
\end{enumerate}
If $f$ is a nonzero central element in $R$, then the subgroup of $G$
generated by $\Supp(f)$ is an FC-group.
\end{lem}

\begin{proof}
Let $f = \sum_{s\in G} f_s$ be a nonzero central element in $R$,
and let $H$ be the subgroup of $G$ that is generated by the set $\Supp(f)$.
Write $\Omega := \Supp(f) \cup \{e\} \cup \Supp(f)^{-1}$.
We begin by showing that $\Omega$ is invariant under conjugation by elements from $\Omega$.
The element $e$ is obviously fixed under conjugation by any element, and any element is fixed under conjugation by $e$.

Take $s,g\in \Supp(f) \setminus \{e\}$.
Consider $r:=f_s$.

\underline{Case A. (i) holds:}

Choose a nonzero $a \in R_g$ such that $a f_s \neq 0$.
Note that $g s \in \Supp(a f) = \Supp(f a)$, because $a f = f a$.
Hence, there is some $t\in \Supp(f)$ such that $g s = t g$.
In other words, $g s g^{-1} = t \in \Supp(f) \subseteq \Omega$.

We claim that $g^{-1} s g \in \Supp(f)$ also holds.
The preceding paragraph shows that $g\Supp(f)g^{-1} \subseteq \Supp(f)$.
Since conjugation by $g$ is injective and $\Supp(f)$ is finite, this inclusion is equality.
Hence, there is some $k > 0$
such that $g^k s g^{-k} = s$. 
If $k=1$, then clearly $g^{-1}s g = s \in \Supp(f) \subseteq \Omega$.
If $k>1$, then we note that $g^{-1} s g = g^{k-1} s g^{-k+1} \in \Supp(f) \subseteq \Omega$.

From $gsg^{-1} \in \Supp(f)$, we get that $g s^{-1} g^{-1} \in \Supp(f)^{-1} \subseteq \Omega$.
Furthermore, from $g^{-1} s g \in \Supp(f)$, we get that $g^{-1} s^{-1} g \in \Supp(f)^{-1} \subseteq \Omega$.

\underline{Case B. (ii) holds:}

An argument completely analogous to that in Case A allows us to conclude that
$g^{-1} s g \in \Supp(f) \subseteq \Omega$ and $g s g^{-1} \in \Supp(f) \subseteq \Omega$.
It then follows that $g^{-1} s^{-1}g  \in \Supp(f)^{-1} \subseteq \Omega$
and $g s^{-1} g^{-1} \in \Supp(f)^{-1} \subseteq \Omega$.

%
%

In both cases, we are able to establish that $\Omega$ is invariant under conjugation by elements from $\Omega$.
Every element of $H$ is a word in $\Omega$.
Thus, $\Omega$ is invariant under conjugation by elements from $H$.

Finally, let $h$ be an arbitrary element of $H$.
Then there is some $m \geq 1$ such that 
$h = s_1 \cdots s_m$ 
with $\{s_1,\ldots,s_m\} \subseteq \Omega$.
For any $w \in H$, we may write
\begin{displaymath}
	w h w^{-1} = (ws_1 w^{-1}) \cdots (w s_m w^{-1}).
\end{displaymath}
By the preceding paragraph, each $ws_i w^{-1}$ lies in the finite set $\Omega$.
Hence, the conjugacy class of $h$ contains at most $|\Omega|^m$ elements.
In particular, it is finite.
Thus, $H$ is an FC-group.
\end{proof}

The following example shows that the two conditions  
in Lemma~\ref{lem:centralFCsupportGeneralized} are independent.

\begin{exmp}
Let $G:=\Z$ and consider the triangular matrix ring
\begin{displaymath}
		R:=\begin{pmatrix}
	\Q	&	\Q[t] \\
	0		&	\Q[t]
\end{pmatrix}
\end{displaymath}
with the usual addition and multiplication. Endow $R$ with a $G$-grading by setting
\begin{itemize}
	\item $R_n:=
\left\{
\begin{pmatrix}
	0 & 0 \\
	0 & 0
\end{pmatrix}\right\}$
for $n<0$,
\item $R_0:=
	\begin{pmatrix}
		\Q & 0 \\
		0				& \Q
	\end{pmatrix}$, and
	\item $R_n := 
\begin{pmatrix}
	0 & \Q t^{n-1} \\
	0 & \Q t^n
\end{pmatrix}$
for $n>0$.
\end{itemize}
It is easy to verify that condition (ii) of Lemma~\ref{lem:centralFCsupportGeneralized} is satisfied.
But note that for the nonzero homogeneous element
$r:=\begin{pmatrix}
	1 & 0 \\
	0 & 0
\end{pmatrix} \in R_0$
we have $R_1 r = \{0\}$.
Thus, this grading fails to satisfy condition (i) of Lemma~\ref{lem:centralFCsupportGeneralized}.
Passing to the opposite ring with the induced $G$-grading gives an example satisfying condition (i) but not condition (ii).
\end{exmp}

We pause for a moment to record a consequence of Lemma~\ref{lem:centralFCsupportGeneralized} 
that is of independent interest.
Recall that a group $G$ is said to be an \emph{ICC group} if 
each non-trivial element of the group has an infinite conjugacy class (see e.g. \cite{deCornulier}). 
Non-trivial abelian groups cannot be ICC. But there exist torsion ICC groups.
From that perspective, the next corollary complements the picture provided by Corollary~\ref{cor:AbelianTorFree}.

\begin{cor}\label{cor:ICC}
Let $G$ be an ICC group and let $R$ be a 
$G$-graded ring with $\Supp(R)=G$.  
Suppose that 
at least one of the following two assertions holds:
\begin{enumerate}[{\rm (i)}]
	\item $R_gr \neq \{0\}$
	for all nonzero homogeneous $r\in R$ and all $g\in G$.
	\item $rR_g \neq \{0\}$
	for all nonzero homogeneous $r\in R$ and all $g\in G$.
\end{enumerate}
If $f$ is a central idempotent in $R$,
then $f \in R_e$.
\end{cor}

\begin{proof}
Suppose that $f$ is a central idempotent in $R$.
If $f=0$, then $f\in R_e$. Now, suppose that $f$ is nonzero.

The same argument as in Lemma~\ref{lem:centralFCsupportGeneralized} applies with $g\in \Supp(R)=G$, since the hypotheses supply a homogeneous element of degree $g$ whose product with $f_s$ is nonzero on the appropriate side.
Thus, the finite set $\Omega$ is invariant with respect to conjugation by every $g \in G$.
Using that $G$ is an ICC group, we conclude that $\Omega$ contains no non-trivial element.
Hence, $\Supp(f) \subseteq \{e\}$, i.e. $f \in R_e$.
\end{proof}

Recall that a group is said to be \emph{locally normal} if every finite subset is contained
in a finite normal subgroup. We recall the following result from \cite[Lem.~1]{Burns}.

\begin{lem}[Burns]\label{lem:FCisomorphic}
Any FC-group $G$ is isomorphic to a subdirect product of a torsion-free
abelian group $A$ with a locally normal group $B$.
\end{lem}

\begin{rem}\label{rem:subdirect}
Suppose that a group $G$ is a subdirect product of $A$ and $B$, for groups $A$ and $B$.
After identifying $G$ with its image in $A \times B$, 
there are surjective group homomorphisms $\pi_A : G \to A, \ (a,b) \mapsto a$
and $\pi_B : G \to B, \ (a,b) \mapsto b$.

Consider $N_A := \ker(\pi_A)$ which is a normal subgroup of $G$.
By the first isomorphism theorem, 
we have $G/N_A \cong A$.
Also note that $N_A$ is isomorphic to a subgroup of $B$. 
\end{rem}

\begin{rem}\label{rem:SwitchGrading}
We recall some standard notation from \cite[Sec.~1.2]{NVO}.
Let $R$ be a $G$-graded ring. For any non-empty subset $X$ of $G$, 
we will write $R_X := \oplus_{x\in X} R_x$.
If $H$ is a subgroup of $G$, then $R_H$ is an $H$-graded subring of $R$.
If $N$ is a normal subgroup of $G$, then we may view $R$ as a $G/N$-graded ring.
For any $C \in G/N$ we set $R_C := \oplus_{x\in C} R_x$.
It is easy to see that $R = \oplus_{C \in G/N} R_C$, and $R_C R_D \subseteq R_{CD}$, for all $C,D \in G/N$.
Note, in particular, that $N$ is the identity element of $G/N$.
When we view $R$ as a $G$-graded ring, the so-called \emph{principal component} will be the subring $R_e$.
But when viewed as a $G/N$-graded ring, the principal component will be the subring $R_N$.
\end{rem}

Next, we prove our second main result.

\begin{proof}[Proof of Theorem~\ref{thm:Main2}]
Let $f$ be a nonzero central idempotent in $R$, and let $H$ be the support group of $f$.
Clearly, $f$ is contained in the $H$-graded subring $R_H$.
By Lemma~\ref{lem:centralFCsupportGeneralized}, $H$ is an FC-group.
Using Lemma~\ref{lem:FCisomorphic}, we can view $H$ as a subdirect product of $A$ and $B$, where $A$ is torsion-free abelian and $B$ is locally normal.
By Remark~\ref{rem:subdirect}, there is a normal subgroup $N_A$ of $H$ such that $H/N_A \cong A$.

By the aforementioned isomorphism and Remark~\ref{rem:SwitchGrading}, we may view $R_H$ as an $A$-graded ring with principal component $R_{N_A}$.
Clearly, $f$ is central in $R_H$.
Hence, by Corollary~\ref{cor:AbelianTorFree}, 
$f$ is contained in $R_{N_A}$ since $A$ is torsion-free abelian.

Note that subgroups of locally normal groups are again locally normal.
Therefore, since $B$ is locally normal, the subgroup $N_A$ is also locally normal.
By the preceding paragraph, $\Supp(f)$ is contained in $N_A$, and hence there is a finite subgroup of $N_A$ containing $\Supp(f)$.
Thus, the support group of $f$ must be finite.
\end{proof}

\begin{rem}
In Theorem~\ref{thm:Main1}, $G$ is abelian, and hence
the finite support group obtained is normal in $G$.
In Theorem~\ref{thm:Main2},
the finite support group obtained is normal in $G$ whenever $\Supp(R)=G$.
Indeed, this is a direct consequence of the proof of Lemma~\ref{lem:centralFCsupportGeneralized}.
In general, however, the support group of a nonzero central idempotent need not be normal in $G$.
To see this, let $G:=S_3$ be the symmetric group on three elements and
consider the subgroup $K$ generated by the permutation $(12)$. 
The group ring $R:=\Q[K]$ has its canonical $K$-grading that can be extended to  
a $G$-grading by setting $R_g := \{0\}$ for $g\in G \setminus K$.
The element
$f:=\frac{1}{2}(1+u_{(12)})$
is a nonzero central idempotent in $R$. Clearly, $K$ is the support group of $f$, but $K$ is not normal in $G$.
\end{rem}


We now present three examples, beginning with a familiar one.

\begin{exmp}
Let $A$ be a unital ring and consider the polynomial ring $R:=A[t]$ in one indeterminate $t$.
Define a $\Z$-grading on $R$ by $R_n:=At^n$ for $n\geq 0$, and $R_n:=\{0\}$ for $n<0$.
By Corollary~\ref{cor:AbelianTorFree}, every central idempotent in $R$ must be contained in $R_0=A$.
\end{exmp}

Burns's theorem, as well as the main results in this article, cannot be generalized to hold for arbitrary idempotents,
as the next example illustrates.

\begin{exmp}\label{ex:NonExample1}
Consider the group ring $R:=M_2(\C)[\Z]$ and let $t$ be a generator for $\Z$.
The element
$f:=
\begin{pmatrix}
1 & 0 \\
0 & 0
\end{pmatrix}
u_e
+
\begin{pmatrix}
0 & 1 \\
0 & 0
\end{pmatrix}
u_t$
is a non-central idempotent in $R$.
We have $\Supp(f)=\{e,t\}$.
Hence, the support group of $f$ is infinite.
\end{exmp}

The next example illustrates a limitation of 
possible generalizations of Theorem~\ref{thm:Main1} and Theorem~\ref{thm:Main2}.

\begin{exmp}\label{ex:Failure1}
Let $G:=D_\infty=\langle s,t \mid s^2=t^2=e\rangle$ be the infinite dihedral group,
and let $R:=\Q^4$ be the commutative direct product $\Q$-algebra.
Consider the four elements
\begin{displaymath}
1_R=(1,1,1,1), \ \
a:=(1,1,-1,-1), \ \
b:=(1,-1,0,0), \ \text{ and } \
c:=(0,0,1,-1).
\end{displaymath}
We define additive subgroups of $R$ as follows:
\begin{displaymath}
R_e:=\Q 1_R \oplus \Q a, \ \
R_s := \Q b, \ \
R_t := \Q c, \ 
\text{ and } \ 
R_g := \{(0,0,0,0)\} \ \text{ for } \ g \in G \setminus \{e,s,t\}.	
\end{displaymath}
Note that $\{1_R,a,b,c\}$ is a basis for $R$ as a $\Q$-vector space.
Thus, $R=\oplus_{g\in G} R_g$.

We shall now make a few observations:
\begin{itemize}
	\item $a^2=1_R \in R_e$, and hence $R_eR_e \subseteq R_e$.
	\item $ab=b$ and $ac=-c$, and hence $R_e R_s \subseteq R_s$ and $R_e R_t \subseteq R_t$.
	\item $b^2=(1,1,0,0)=\frac{1}{2}(1+a) \in R_e$, and hence $R_s R_s \subseteq R_e = R_{s^2}$.
	\item $c^2=(0,0,1,1)=\frac{1}{2}(1-a) \in R_e$, and hence $R_t R_t \subseteq R_e = R_{t^2}$.
	\item $bc=0$, and hence $R_s R_t = \{0\} \subseteq R_{st}$.
\end{itemize}
We conclude that $R_g R_h \subseteq R_{gh}$ for all $g,h \in G$. This shows that $R$ is $G$-graded.
Note, however, that $R_s c = \{0\}$ and $c R_s= \{0\}$. Hence, neither assertion (i) nor assertion (ii) in Theorem~\ref{thm:Main2} is  satisfied. Furthermore, $D_\infty$ is non-abelian and thus violates the requirement in Theorem~\ref{thm:Main1}.

Now, consider the element 
\begin{displaymath}
f:=\frac{1}{2}1_R+\frac{1}{2}b+\frac{1}{2}c
=\frac{1}{2}(1,1,1,1)+\frac{1}{2}(1,-1,0,0)+\frac{1}{2}(0,0,1,-1)
=(1,0,1,0).
\end{displaymath}
Note that $f$ is a nonzero central idempotent in $R$.
But the subgroup generated by $\Supp(f)=\{e,s,t\}$ is equal to the infinite group $D_\infty$.
\end{exmp}

The following lemma generalizes \cite[Lem.~2.10(b)]{LLOW}.

\begin{lem}\label{lem:SunitsTrick}
Let $G$ be a group and let $R$ be an $s$-unital $G$-graded ring.
Then $r \in r R_e \cap R_e r$ for every $r\in R$.
\end{lem}

\begin{proof}
Take $r\in R$.
If $r=0$, then there is nothing to prove.
Suppose therefore that $r\neq 0$.
Write $\Supp(r)=\{g_1,\ldots,g_n\}$.
Consider the finite set $\{r_{g_1},\ldots,r_{g_n}\} \subseteq R$.
By $s$-unitality and \cite{Tominaga} (see also \cite[Prop.~2.1]{LLOW}), there is some $a \in R$ such that
$a r_{g_i} = r_{g_i} a = r_{g_i}$ for every $i\in \{1,\ldots,n\}$.
Note that $a_e r_{g_i} = r_{g_i} a_e = r_{g_i}$ for every $i\in \{1,\ldots,n\}$.
From this we get $a_e r = r a_e = r$.
In particular, $r \in r R_e \cap R_e r$.
\end{proof}

We end this section by proving the last three corollaries from the introduction.

\begin{proof}[Proof of Corollary~\ref{cor:Strongly}]
Suppose that
$r\in R$ satisfies
$R_g r = \{0\}$ for some $g\in G$.
Then, by Lemma~\ref{lem:SunitsTrick} and the strong grading, we have
$r \in R_e r = R_{g^{-1}} R_g r = \{0\}$.
The first part of the claim now follows from Theorem~\ref{thm:Main2}.
If $G$ is torsion-free, then every finite support group is trivial,
and hence all central idempotents will be contained in $R_e$.
\end{proof}

\begin{rem}
Mihovski \cite{Mihovski} considered the unital case of Corollary~\ref{cor:Strongly} using a different method.
His argument is particularly well suited for $G$-crossed products over $\Q$-algebras.
However, in the full generality of unital strongly $G$-graded rings, 
certain steps in the proof of \cite[Lem.~1]{Mihovski}
do not appear to apply directly without further hypotheses;
this concerns, in particular, the initial claim and the Taylor expansion on \cite[p.~106]{Mihovski}.
\end{rem}

\begin{proof}[Proof of Corollary~\ref{cor:Crystalline}]
A crystalline graded ring $A*G$, as defined in \cite{NauwVO}, carries a natural $G$-grading.
It is clear from the definition that for every $g\in G$, there is a nonzero element in $(A*G)_g$ that is 
neither a left nor a right zero-divisor in $A*G$ (see \cite[Cor.~2.7]{NauwVO}).
The first part of the claim now follows from Theorem~\ref{thm:Main2}.
If $G$ is torsion-free, then every finite support group is trivial,
and hence all central idempotents will be contained in $(A*G)_e=A$.
\end{proof}

\begin{proof}[Proof of Corollary~\ref{cor:NonDeg}]
Suppose that $R$ is right non-degenerately $G$-graded. 
Seeking a contradiction, 
suppose that there is a nonzero element $r\in R_h$, for some $h\in G$, 
such that $R_g r = \{0\}$ for some $g\in \Supp(R)$.
Consider the sets $I:=R_{g^{-1}} R_g$ and $J:=r R_{h^{-1}}$ which are right ideals of $R_e$.
Clearly, by right non-degeneracy of the grading, $J$ is nonzero.
By right non-degeneracy of the grading it also follows that $R_{g^{-1}}$ is nonzero,
and after using right non-degeneracy of the grading yet again, we see that $I$ must be nonzero. 
Note that $IJ=R_{g^{-1}} R_g r R_{h^{-1}} = \{0\}$.
But $R_e$ is prime, and in a prime ring a product of two nonzero right ideals is nonzero.
This is a contradiction. 
The proof in the left non-degenerately $G$-graded case is analogous. 
The first part of the claim now follows from Theorem~\ref{thm:Main2}.
If $G$ is torsion-free, then every finite support group is trivial,
and hence all central idempotents will be contained in $R_e$.
\end{proof}

\section{Applications}\label{Sec:Applications}

In this final section, we apply our results to various types of graded rings.

\subsection{Semigroup-graded rings}

Recall that a semigroup $S$
is said to be \emph{right reversible} if $Sa \cap Sb \neq \emptyset$ for all $a$ and $b$ in $S$.
Left reversibility is defined analogously.
Ore has shown that any right reversible cancellative semigroup can be embedded in a group (see \cite[Thm.~1.23]{CliffordPreston}).
The same conclusion can be shown to hold for left reversible cancellative semigroups.

The \emph{support semigroup} of a nonzero element $r$ in an $S$-graded ring is the subsemigroup of $S$ generated by $\Supp(r)$.

\begin{prop}\label{prop:semigroups}
Let $S$ be a cancellative semigroup and let $R$ be an $S$-graded ring.
Furthermore, suppose that at least one of the following assertions holds:
\begin{enumerate}[{\rm (a)}]
	\item $S$ is commutative.
	\item $S$ is right (or left) reversible, and 
at least one of the following two assertions holds:
\begin{enumerate}[{\rm (i)}]
	\item $R_gr \neq \{0\}$
	for all nonzero homogeneous $r\in R$ and all $g\in \Supp(R)$.
	\item $rR_g \neq \{0\}$
	for all nonzero homogeneous $r\in R$ and all $g\in \Supp(R)$.
\end{enumerate}
\end{enumerate}

\noindent If $f$ is a nonzero central idempotent in $R$,
then $f$ has finite support semigroup in $S$.
\end{prop}

\begin{proof}
Since $S$ is cancellative, and any commutative semigroup is right (and left) reversible,
it follows from \cite[Thm.~1.23]{CliffordPreston}, in both cases (a) and (b), that
$S$ embeds in a group $G$.

Now, $R$ is already $S$-graded. For each $g\in G \setminus S$, we set $R_g:=\{0\}$.
We may now view $R$ as a $G$-graded ring.
Let $f$ be a nonzero central idempotent in $R$.

In case (a), the group $G$ can in fact be assumed to be abelian.
By Theorem~\ref{thm:Main1}, the support group of $f$ in $G$ is finite.
Hence, the support semigroup of $f$ in $S$ must also be finite.

In case (b), we may apply Theorem~\ref{thm:Main2} to reach the same conclusion.
\end{proof}

\subsection{Leavitt path rings}

For the definition of Leavitt path algebras and Leavitt path rings, we refer the reader to \cite{AbramsEtAl} and \cite[Sec.~4]{LLOW2}, respectively.
The following proposition is an immediate consequence of Corollary~\ref{cor:AbelianTorFree}.

\begin{prop}
Let $A$ be a unital ring, let $E$ be a directed graph, and let $R:=L_A(E)$ be the corresponding Leavitt path ring.
Equip $R$ with its canonical $\Z$-grading.
If $f$ is a central idempotent in $R$, then $f\in R_0$.
\end{prop}

\subsection{Fractional skew monoid rings}

For the definition of fractional skew monoid rings, we refer the reader to \cite{AraEtAl}.
The next proposition follows from Theorem~\ref{thm:Main1}.

\begin{prop}
Let $A$ be a unital ring, let $G$ be an abelian group, and let $S$ be a submonoid of $G$ such that $G=S^{-1}S$.
Furthermore, let $\Endr(A)$ be the monoid of not necessarily unital ring endomorphisms of $A$,
and let $\alpha : S \to \Endr(A)$ be a monoid homomorphism.
Consider the corresponding fractional skew monoid ring $R:=S^{\mathrm{op}} *_\alpha A *_\alpha S$ with its canonical $G$-grading. 
If $f$ is a nonzero central idempotent in $R$, then $f$ has finite support group.
Moreover, if $G$ is torsion-free, then every central idempotent lies in $R_e$.
\end{prop}

\subsection{Partial skew group rings}

For the definition of partial skew group rings, we refer the reader to \cite[Sec.~11]{LLOW}.
The following proposition is an immediate consequence of Corollary~\ref{cor:AbelianTorFree}.

\begin{prop}
Let $G$ be a torsion-free abelian group and let $(\{\alpha_g\}_{g\in G}, \{D_g\}_{g\in G})$ be a partial action of $G$ on an $s$-unital ring $A$, with $D_g$ $s$-unital for every $g\in G$.
Every central idempotent in the partial skew group ring $A \star_\alpha G$ is contained in $A$.
\end{prop}

The next proposition follows from Corollary~\ref{cor:NonDeg}, \cite[Prop.~11.1]{LLOW}, and \cite[Prop.~2.9]{LLOW}.

\begin{prop}
Let $G$ be a group and let $(\{\alpha_g\}_{g\in G}, \{D_g\}_{g\in G})$ be a partial action of $G$ on an $s$-unital prime ring $A$, with $D_g$ $s$-unital for every $g\in G$.
Every nonzero central idempotent in the partial skew group ring $A \star_\alpha G$ has finite support group.
\end{prop}

\subsection{Algebraic Cuntz-Pimsner rings}

For the definition of algebraic Cuntz-Pimsner rings, we refer the reader to \cite{CarlsenOrtega}.
Recall that the Toeplitz ring $\mathcal{T}_{(P,Q,\psi)}$ 
and the relative Cuntz-Pimsner ring $\mathcal{O}_{(P,Q,\psi)}(J)$
carry natural $\Z$-gradings, and that the image of the coefficient ring is contained in degree zero.
The next proposition follows from Corollary~\ref{cor:AbelianTorFree}.

\begin{prop}
Let $R$ be a ring and let $(P,Q,\psi)$ be an $R$-system.
Then every central idempotent in the Toeplitz ring $\mathcal{T}_{(P,Q,\psi)}$ lies in $(\mathcal{T}_{(P,Q,\psi)})_0.$
If, moreover, $(P,Q,\psi)$ satisfies Condition (FS) and $J$ is a $\psi$-compatible ideal of $R$,
then every central idempotent in the relative Cuntz-Pimsner ring $\mathcal{O}_{(P,Q,\psi)}(J)$ lies in $(\mathcal{O}_{(P,Q,\psi)}(J))_0$.
\end{prop}

\end{document}